\documentclass[pdflatex,sn-mathphys-num]{sn-jnl}

\usepackage{graphicx}%
\usepackage{multirow}%
\usepackage{amsmath,amssymb,amsfonts}%
\usepackage{amsthm}%
\usepackage{mathrsfs}%
\usepackage[title]{appendix}%
\usepackage{xcolor}%
\usepackage{textcomp}%
\usepackage{manyfoot}%
\usepackage{booktabs}%
\usepackage{algorithm}%
\usepackage{algorithmicx}%
\usepackage{algpseudocode}%
\usepackage{listings}%
\usepackage{float}

\newcommand{\bfU}{{\boldsymbol U}}

\newcommand{\DG}{\mathrm{DG}}

\theoremstyle{thmstyleone}%
\newtheorem{theorem}{Theorem}
\newtheorem{lemma}[theorem]{Lemma}%

\theoremstyle{thmstyletwo}%

\theoremstyle{thmstylethree}%

\begin{document}

\title[Article Title]{Numerical Analysis of a Coupled 3D-1D Transport Problem}


\author*[1]{\fnm{Alyssa M.} \sur{Taylor-LaPole}}\email{alyssa.lapole@rice.edu}

\author[1]{\fnm{Uzochi} \sur{Gideon}}\email{uzochi.gideon@rice.edu}

\author[1]{\fnm{Beatrice} \sur{Riviere}}\email{riviere@rice.edu}

\author[1]{\fnm{Duygu} \sur{Vargun}}\email{duygu.vargun@rice.edu}

\affil[1]{\orgdiv{Department of Computational Applied Mathematics \& Operations Research}, \orgname{Rice University}, \orgaddress{\street{6100 Main Street}, \city{Houston}, \postcode{77005}, \state{TX}, \country{USA}}}


\abstract{A finite element solution coupled with an interior penalty discontinuous Galerkin solution are defined for the approximation of the coupled 3D-1D solute transport problem. 
Under sufficient regularity for the weak solutions, optimal error bounds are derived for the 3D concentration and 1D concentration, that are optimal with respect to the time step size and the mesh sizes. Numerical results verify the theoretical results.}




\maketitle


\section{Introduction}\label{sec1}



 Coupled 3D–1D partial differential equations are primarily used to model multiscale problems in which  thin, lower-dimensional structures are embedded within a three-dimensional bulk domain. These models are applied across various scientific and engineering disciplines, ranging from biomedical and physiological sciences \cite{cattaneo2014computational, notaro2016method,koppl20203d,fritz20221d} to geosciences and material sciences \cite{gjerde2020singularity, malenica2018groundwater,cerroni2019mathematical,gjerde2018well,llau2016concrete}, to address problems characterized by significant disparities in the scales of the domains that constitute the model.
Mixed dimensional models avoid the prohibitively high cost of creating a 3D mesh fine enough to resolve the tiny radius of a capillary or well within a large volume, by reducing thin 3D inclusions to 1D manifolds via topological model reduction techniques proposed \cite{dAngelo2008coupling,Zunino19}.



In this work, we formulate and analyze a numerical method to solve the coupled 3D-1D time-dependent convection diffusion equations, introduced in \cite{masri2024modelling}. Solute exchange between the 1D and 3D domains is made possible through a flux that is proportional to a permeability function and the difference between the
1D concentration and the lateral average of the 3D concentration.  Our scheme combines the backward Euler method in time with the finite element method and with the interior penalty discontinuous Galerkin method in space. Optimal error bounds are derived for the concentration of the solute.
Our analysis is valid in the general case where the permeability of the inclusion wall is a non negative function along the 1D domain; this allows for parts of the wall of the blood vessel  (or wall of the inclusion) to be impermeable. To our knowledge, our work is the first error analysis of the 3D-1D transport problems. 

The numerical analysis of coupled 3D-1D flows is introduced in \cite{Zunino19} where the elliptic problem is discretized by the finite element method. Convergence is obtained by deriving a priori error estimates.  The authors derive the 3D-1D mixed dimensional model from
the fully 3D-3D model and they show that the 3D-1D model has a modeling error that depends on the radius of the tubular inclusion.  The coupled 3D-1D elliptic problem is discretized by the interior penalty discontinuous Galerkin method in \cite{Masri24}. Because of the coupling term involving the 1D centerline, the 3D solution exhibits low regularity and consistency of the DG scheme does not hold.
The case of numerical discretization of the time-dependent coupled 3D-1D problems has not been studied in the literature. As mentioned above, the work \cite{masri2024modelling}  contains the derivation of the 3D-1D transport problem by starting from the 3D-3D models of solute transport and by applying a similar topological reduction to the one in \cite{Zunino19}.  The authors also show that the modeling error decreases with the inclusion diameter.

An outline of the paper is as follows. We first introduce the model problem and the scheme in Section~\ref{sec:model} and Section~\ref{sec:formulation}. Existence, uniqueness
and stability bounds are obtained in Section~\ref{sec:well}.  Convergence is proved by deriving error bounds in the following section.  Numerical results are given in Section~\ref{sec:numerics} and conclusions follow.

\section{Model problem}
\label{sec:model}
\noindent Let $\Omega\subset\mathbb{R}^3$ be a bounded Lipschitz domain and let $\Lambda$ be $\mathcal{C}^2$ curve entirely contained in $\Omega$. We parametrize  $\Lambda$ by the function $\lambda(s)$ with $s\in[0,L]$.  We assume $\lambda$ to be the centerline of a generalized cylinder $\Sigma_\Lambda$ with boundary $\Gamma$. Let $\gamma$ be the permeability of the boundary. 
At each point $s$ on $\Lambda$, the cross-section is denoted by $\mathcal{D}(s)$ with  area $|\mathcal{D}(s)|$ and circumference $|\partial\mathcal{D}(s)|$. 

We model the transport of a solute by defining $c$ and $\hat c$ the concentrations of the solute in $\Omega$ and $\Lambda$ respectively over the time interval $[0,T]$. The mass balance equations are:
\begin{align}
\frac{\partial c}{\partial t} & - \nabla \cdot (\kappa \nabla c) +\nabla \cdot (\bfU c) 
+ \gamma (\overline{c} - \hat{c}) \delta_\Lambda = f, \,\, \mbox{in}\,\Omega\times (0,T),\label{eq:3D}\\
|\mathcal{D}(s)|\frac{\partial \hat c}{\partial t} & - \frac{\partial}{\partial s}\left(|\mathcal{D}(s)|\hat \kappa \frac{\partial\hat c}{\partial s}\right) +\frac{\partial}{\partial s}\left(|\mathcal{D}(s)| \hat U \hat c\right) + |\partial\mathcal{D}(s)|\gamma (\hat c-\overline{c}) 
= \hat f, \,\,\mbox{in}\,\Lambda\times (0,T).\label{eq:1D}
\end{align}
The interaction between the 1D curve and the 3D domain is done with the term $\gamma (\overline{c}-\hat c)\delta_\Lambda$, where $\delta_\Lambda$ is the Dirac function on $\Lambda$. It is to be understood in the weak sense as follows: for any function $v\in H^1(\Omega)$:
\[
\int_\Omega \gamma (\overline{c}-\hat c)\delta_\Lambda v
= \int_{\Lambda} \vert \partial \mathcal{D}(s)\vert (\overline{c}-\hat c) \overline{v}.
\]
The functions $\overline{c}$ and $\overline{v}$ denote the lateral average of $c$ and $v$ respectively. For any $0\leq s \leq L$ and any function $w\in L^1(\partial\mathcal{D}(s))$, we define
\[
\overline{w}(s) = \frac{1}{\vert \partial \mathcal{D}(s)\vert} \int_{\partial\mathcal{D}(s)} w.
\]
In the model above, $\kappa$ (resp. $\hat \kappa$) denote the diffusion coefficient in $\Omega$ (resp. $\Lambda$), $f$ (resp. $\hat f$) denote the source/sink function in $\Omega$ (resp. $\Lambda$), and $\bfU$
(resp. $\hat U$) denote the velocity field in $\Omega$ (resp. $\Lambda$).

The initial conditions are:
\begin{equation}\label{eq:initC}
c(\cdot,0) = c^0(\cdot), \quad \mbox{in}\, \Omega,\quad
\hat c(\cdot,0) = \hat c^0(\cdot), \quad \mbox{in}\,
\Lambda.
\end{equation}
We impose zero Dirichlet boundary conditions on the 3D domain
\begin{equation}
c = 0 \text{ on } \partial\Omega\times (0,T),
\end{equation}
and we impose inflow and outflow boundary conditions at the inlet ($s=0$) and outlet ($s=L$) of $\Lambda$. At the inlet, $c^\textrm{in}$ denotes a prescribed concentration.
\begin{align}
    |\mathcal{D}|\hat \kappa \frac{\partial\hat c}{\partial s} - |\mathcal{D}| \hat U \hat c = 
    - |\mathcal{D}| \hat U   c^{\mathrm{in}}, \quad
    \mbox{at} \quad s=0, \label{eq:bchat1}\\
    |\mathcal{D}|\hat \kappa \frac{\partial\hat c}{\partial s}  = 0, \quad
    \mbox{at} \quad s=L.
    \label{eq:bchat2}
\end{align}

The following assumptions are made on the data.
\begin{enumerate}
    \item[(i)] The generalized cylinder does not collapse: there are positive constants $d_0$ and $d_1$ such that
\begin{equation}\label{eq:Dbounded}
    \forall 0\leq s \leq L, \quad
    d_0 \leq \vert \mathcal{D}(s) \vert \leq d_1, \quad
    d_0 \leq \vert \partial\mathcal{D}(s) \vert \leq d_1.
    \end{equation}
    The function $s\mapsto \vert \mathcal{D}(s)\vert $ is continuous and piecewise differentiable with non-negative derivative.
    If $\Sigma_\Lambda$ is a cylinder with  constant radius, these conditions are trivially satisfied.
    \item[(ii)] The functions $\kappa, \hat \kappa$ are $L^\infty$ functions, bounded above and below by positive constants. The function $\gamma\in L^\infty(\Gamma)$ is non-negative and bounded.
    \[
    0< k_0 \leq \kappa(x) \leq k_1, \quad
    0< \hat k_0 \leq \hat \kappa (s) \leq \hat k_1, \quad
    0\leq \gamma \leq \gamma_1.
    \]
    \item[(iii)] The velocities are independent of time. In $\Lambda$, the velocity $\hat U>0$ is a positive constant. In $\Omega$, we make some assumptions on  the advection field $\bfU$, that depend on the divergence of $\bfU$.  For the analysis below, we differentiate the two cases by introducing a parameter $\eta_\bfU = 0,1$. 
    \begin{align}
    \mbox{Case 1:}&\quad \nabla \cdot \bfU = 0, \quad \eta_\bfU = 0,\\
    \mbox{Case 2:}&\quad \Vert \bfU \Vert_{L^\infty(\Omega)} \leq \frac{k_0}{2 C_0}, \quad \eta_\bfU = 1.\label{eq:condU}
    \end{align}
    The constant $C_0$ is the Poincar\'e's constant:
    \begin{equation}
\label{eq:poincare}
\forall v\in H^1_0(\Omega), \quad 
\Vert v \Vert_{L^2(\Omega)} \leq C_0 \Vert \nabla v \Vert_{L^2(\Omega)}.
\end{equation}
    \item[(iv)] The source/sink functions  and prescribed concentration are continuous in time: $f\in \mathcal{C}(0,T;L^2(\Omega))$, $\hat f \in \mathcal{C}(0,T;L^2(\Lambda))$ and $c^\mathrm{in}\in \mathcal{C}(0,T)$.
    \item[(v)] The initial conditions satisfy:  $c^0 \in H^1(\Omega)$   and $\hat c^0 \in L^2(\Lambda)$.
\end{enumerate}


A weak solution $(c,\hat c)$ is defined such that
$(c,\hat c)$ satisfies \eqref{eq:initC}, 
$c\in L^2(0,T;H_0^1(\Omega))$, $\hat c \in L^2(0,T;H^1(\Lambda))$ with $\partial_t c \in L^2(0,T;L^2(\Omega))$, $\partial_t \hat c \in L^2(0,T;L^2(\Lambda))$ and such that  for all $v\in H_0^1(\Omega)$ and $\hat v \in H^1(\Lambda)$:
\begin{align*}
    \int_\Omega\frac{\partial c}{\partial t}  v
      +   \int_\Omega  \kappa \nabla c \cdot \nabla v 
      - \int_\Omega \bfU c \cdot \nabla v
      +\int_\Lambda\gamma \vert \partial\mathcal{D}\vert \left(\overline{c}-\hat{c}\right)\overline{v}
        = \int_\Omega f v,\\
    \int_\Lambda |\mathcal{D}| \frac{\partial \hat c}{\partial t} \hat v
    + \int_\Lambda \vert \mathcal{D}\vert \hat \kappa \frac{\partial \hat c}{\partial s} \frac{\partial \hat v}{\partial s} 
    -\int_{\Lambda} |\mathcal D|\,\hat U\,\hat c \,\frac{\partial \hat v}{\partial s}
    +\int_\Lambda \gamma|\partial\mathcal{D}|\left(\hat c-\overline{c}\right)   \hat v \nonumber\\
    + |\mathcal{D}(L)|\hat U \, \hat c(L) \hat v(L)
   =     |\mathcal{D}(0)|\hat U \, c^{\text{in}} \hat v(0) + \int_\Lambda \hat f \hat v.
\end{align*}
The proof of existence and uniqueness of a weak solution follows a similar argument to the one in \cite{masri2024modelling} (Proposition 4.2).

\section{Numerical scheme}
\label{sec:formulation}
\noindent The proposed scheme employs the finite element method to solve \eqref{eq:3D} and the interior penalty discontinuous Galerkin (IPDG) method to solve \eqref{eq:1D}. The time-stepping method is the backward Euler method.
The scheme is defined in this section and its derivation is given in Appendix~\ref{sec:appb}.

Let $\tau>0$ denote the time step value and
let $t^n = n \tau$ denote the discrete time value for any $n\geq 0$.
The domain $\Omega$ is partitioned into tetrahedra elements $ K \in \mathcal{T}^h_\Omega$ . The mesh size is defined by $h_\Omega = \max\limits_{K \in \mathcal{T}^h_\Omega} \mathrm{diam}(K)$. Let $V^\Omega_h$ be the finite dimensional space of continuous piecewise polynomials of degree $k_1\geq 1$ that vanish on the boundary:
\[
V^\Omega_h = \{v_h \in H_0^1(\Omega), \, v_h \in \mathbb{P}_{k_1}(K) \ \forall K \in \mathcal{T}^h_\Omega\}.
\]
Using the parametrization $\lambda$, we partition the curve $\Lambda$ by defining a partition $0=s_0<s_1<\dots<s_N=L$ denoted by $\mathcal{T}_\Lambda^h$,  with mesh size $h_\Lambda$.  
Let $H^1(\mathcal{T}^h_\Lambda)$ be the space of broken $H^1$ functions on $\Omega$ and $V^\Lambda_h$ be the finite dimensional space of discontinuous piecewise polynomials of degree $k_2:$
\[
V^\Lambda_h = \{\hat{v}_h \in L^2(\Lambda), \, \hat{v}_h \in \mathbb{P}_{k_2}((s_{i-1},s_i)) \ \forall i = 1, .., N\}.
\] 
The jump and average operators are defined across interior nodes $s_i$ for $1\leq i \leq N-1$ as, for any function $\hat v\in H^1(\mathcal{T}_\Lambda^h)$:
\begin{align*}
  \hat v(s_i^-) =  \lim_{\epsilon \rightarrow0, \epsilon>0} \hat v(s_i-\epsilon),\ \   \hat v(s_i^+) =  \lim_{\epsilon \rightarrow0, \epsilon>0} \hat v(s_i+\epsilon),\\
[\hat v]_{s_i}= \hat v(s_i^-) -\hat v(s_i^+), \quad \{\hat v\}_{s_i}=\frac{1}{2}\left(\hat v(s_i^-) +\hat v(s_i^+)\right).
\end{align*}
Let $\sigma>0$ denote the penalty parameter used in the IPDG scheme. 
We define a semi-norm for the space $V_h^\Lambda$:
\[
\vert \hat v_h \vert_{\DG} = \left(
\sum_{i=1}^N \Vert \frac{\partial \hat v_h}{\partial s}\Vert_{L^2([s_{i-1},s_i])}^2
+ \frac{\sigma}{h_\Lambda} \sum_{i=1}^{N-1} [\hat v_h]_{s_i}^2 \right)^{1/2}.
\]
We now present the scheme. 
For all $n\geq 1$, given $(c_h^{n-1}, \hat c_h^{n-1})\in V_h^\Omega\times V_h^\Lambda$, find $(c_h^{n}, \hat c_h^{n}) \in V_h^\Omega\times V_h^\Lambda$ such that:
\begin{align}
    \frac{1}{\tau}\int_\Omega (c_h^{n}-c_h^{n-1}) v_h
      +  & \int_\Omega  \kappa \nabla c_h^{n} \cdot \nabla v_h  
       - \int_\Omega \bfU c_h^{n} \cdot \nabla v_h
   \nonumber\\
      +&\int_\Lambda\gamma \vert \partial\mathcal{D}\vert \left(\overline{c_h^{n}}-\hat{c}_h^{n}\right)\overline{v_h}
        = \int_\Omega f(\cdot,t^{n}) v_h, \quad \forall v_h \in V_h^\Omega,\label{eq:FEMscheme}\\
    \frac{1}{\tau} \int_\Lambda |\mathcal{D}|(\hat c_h^{n}- \hat c_h^{n-1}) \hat v_h 
    + & a_\Lambda(\hat c_h^{n}, \hat v_h)
    + b_\Lambda(\hat c_h^{n}, \hat v_h)
    +\int_\Lambda \gamma|\partial\mathcal{D}|\left(\hat c_h^{n}-\overline{c_h^{n}}\right)   \hat v_h
    \nonumber\\
   =  &  \int_\Lambda \hat f(\cdot, t^{n}) \hat v_h
   + |\mathcal{D}(0)|\hat U \, c^{\mathrm{in}}(t^n) \hat v_h(0),
   \quad \forall \hat v_h \in V_h^\Lambda,\label{eq:DGscheme}
\end{align}
where $a_\Lambda(\cdot,\cdot)$ and $b_\Lambda(\cdot,\cdot)$  are defined as follows for any $\hat \xi, \hat v$:
\begin{align*}
  a_\Lambda(\hat \xi, \hat v) = & \sum_{i=1}^N\int_{s_{i-1}}^{s_i}|\mathcal{D}|\hat\kappa \frac{\partial\hat \xi}{\partial s}\frac{\partial \hat v}{\partial s}
    - \sum_{i=1}^{  N-1 } \{|\mathcal{D}|\hat \kappa\frac{\partial \hat \xi}{\partial s}\}_{s_i} \, [\hat v]_{s_i} 
    \\ & -   \epsilon   \sum_{i=1}^{  N-1  } \{|\mathcal{D}|\hat \kappa \frac{\partial \hat v}{\partial s}\}_{s_i} \, [\hat \xi]_{s_i}
    +\sum_{i=1}^{N-1} \frac{\sigma}{h_\Lambda}[\hat \xi]_{s_i} \, [\hat v]_{s_i},  \\  
   b_\Lambda(\hat \xi, \hat v) = &
    -\sum_{i=1}^{N} \int_{s_{i-1}}^{s_i}|\mathcal{D}|\hat U \, \hat \xi \frac{\partial \hat v}{\partial s} 
    +\sum_{i=1}^{  N-1 } \vert\mathcal{D}(s_i)\vert
    \hat U \hat \xi(s_i^-)  [\hat v]_{s_i} 
     +|\mathcal{D}(L)| \hat U \, \hat \xi(L) \hat v(L).
\end{align*}
We remark that $a_\Lambda$ is the standard IPDG discretization of the diffusion operator and $b_\Lambda$ is a slightly modified discretization of the advection operator, that takes into account the weight $\vert \mathcal{D}\vert$.  The parameter $\epsilon$ in $a_\Lambda$ takes the value $-1$, $0$ or $+1$ to yield the classical variations of the IPDG method.

To start the scheme, we choose
\[
c_h^0 = \Pi c^0, \quad \hat c_h^0 = \pi \hat c^0,
\]
where $\Pi c^0$ is the Scott-Zhang interpolant of $c^0$ \cite{ScottZhang} and $\pi \hat c^0$ is the L$^2$ projection of $\hat c^0$ onto $V_h^\Lambda$. 

Throughout the paper, $M$ will denote a generic positive constant that is independent of $h$ and $\tau$ and takes different values at different places.

\section{Well-posedness}
\label{sec:well}

Before we state and prove existence and uniqueness of the discrete solution, we present some important properties of the bilinear forms $a_\Lambda$ and $b_\Lambda$. Throughout the paper, we assume that if $\epsilon = 1, 0$ the penalty parameter $\sigma$ is large enough, and if $\epsilon = -1$ the penalty parameter is chosen $\sigma = 1$. Under this assumption, coercivity of $a_\Lambda$ holds (see proof in Appendix~\ref{sec:appd}): there is a constant $C_a>0$ such that 
\begin{equation}\label{eq:coercivity}
\forall \hat v \in V_h^\Lambda, \quad
C_a \vert \hat v\vert_{\DG}^2 \leq a_{\Lambda}(\hat v, \hat v).
\end{equation}
The proof of the positivity of $b_\Lambda$ is provided in Appendix~\ref{sec:appc}. 
\begin{equation}
    \forall \hat v \in V_h^\Lambda, \quad 
\frac12 \sum_{i=1}^{N-1}  \vert \mathcal{D}(s_i)\vert \, \hat U  [\hat v]_{s_i}^2
+ \frac12 \vert \mathcal{D}(0)\vert \hat U (\hat v(0))^2
 +\frac12 \vert \mathcal{D}(L)\vert \hat U (\hat v(L))^2
\leq b_{\Lambda}(\hat v, \hat v).\label{eq:bpositivity}
\end{equation}
\begin{lemma}
    For each $n\geq 1$, there exists a unique solution $(c_h^{n}, \hat c_h^{n})\in V_h^\Omega\times V_h^\Lambda$ satisfying \eqref{eq:FEMscheme}-\eqref{eq:DGscheme}.   In addition, 
    if $\tau \leq 1/2$, there is a constant $M>0$ independent of $h$ and $\tau$ such that
\begin{align}
\forall m\geq 1, \quad 
    \Vert c_h^m\Vert_{L^2(\Omega)}^2 
    & +d_0 \Vert \hat c_h^n\Vert_{L^2(\Lambda)}^2 
      +  \tau k_0 \sum_{n=1}^m \Vert \nabla c_h^n \Vert_{L^2(\Omega)}^2
      + \tau C_a \sum_{n=1}^m \vert \hat c_h^n \vert_{\DG}^2 \nonumber\\
        \leq  
        & \, M \tau \sum_{n=1}^m \Vert f(\cdot,t^n)\Vert_{L^2(\Omega)}^2
   + M \tau \sum_{n=1}^m \Vert \hat f(\cdot,t^n) \Vert_{L^2(\Lambda)}^2 
   \nonumber\\
&   + M |\mathcal{D}(0)|\hat U \, \tau \sum_{n=1}^{m}(c^{\mathrm{in}}(t^n))^2
+ C \Vert c^{0}\Vert_{H^s(\Omega)}^2 
+ M \Vert \hat c^{0}\Vert_{H^r(\Lambda)}^2.
\label{eq:stability}
\end{align}
\end{lemma}
\begin{proof}
Assume that $n\geq 1$.
Since the problem is linear and finite-dimensional, it suffices to show uniqueness of the solutions.  Assume that $(w,\hat w) = (c_h^{n,1}-c_h^{n,2}, \hat c_h^{n,1}-\hat c_h^{n,2})$, is the pair of differences between two solution pairs. It suffices to show $(w,\hat w) = (0,0)$. By linearity, we have for all $(v_h,\hat v_h)\in V_h^\Omega\times V_h^\Lambda$
    \begin{align*}
    \frac{1}{\tau}\int_\Omega w v_h
      +   \int_\Omega  \kappa \nabla w \cdot \nabla v_h  
      - \int_\Omega \bfU w \cdot \nabla v_h
      +\int_\Lambda\gamma \vert \partial\mathcal{D}\vert \left(\overline{w}-\hat{w}\right)\overline{v_h}
        = 0,\\
      \frac{1}{\tau} \int_\Lambda |\mathcal{D}|\hat w \hat v_h 
    + a_\Lambda(\hat w, \hat v_h)
    + b_\Lambda(\hat w, \hat v_h)
    +\int_\Lambda \gamma|\partial\mathcal{D}|\left(\hat w-\overline{w}\right)   \hat v_h
   =    0.
\end{align*}
We add the two equations and choose $(v_h,\hat v_h) = (w, \hat w)$. Using the coercivity of $a_\Lambda$ and the positivity of $b_\Lambda$, we have
\begin{align*}
    \frac{1}{\tau}\Vert w \Vert_{L^2(\Omega)}^2
    + \Vert \kappa^{1/2} \nabla w \Vert_{L^2(\Omega)}^2  
      - \int_\Omega \bfU w \cdot \nabla w
      +\Vert \gamma^{1/2} \vert \partial\mathcal{D} 
      \vert^{1/2}  (\overline{w}-\hat{w})\Vert_{L^2(\Lambda)}^2
    \\
      +\frac{1}{\tau} \Vert  |\mathcal{D}|^{1/2} \hat w \Vert_{L^2(\Lambda)}^2 
    + C_a \vert \hat w \vert_{\DG}^2
    \leq 0.
\end{align*}
In the case where $\nabla \cdot \bfU = 0$, we have
\[
\int_\Omega \bfU w \cdot \nabla w = 0,
\]
so the third term above disappears.  It is then easy to conclude that $w=\hat w = 0$ by using the coercivity of $a_\Lambda$ and the positivity of $b_\Lambda$.
In the case where $\bfU$ is non-divergence free, we 
write with Poincar\'e's inequality
\begin{align*}
\int_\Omega \bfU w \cdot \nabla w \leq & \Vert \bfU \Vert_{L^\infty(\Omega)} \Vert w \Vert_{L^2(\Omega)} \Vert \nabla w \Vert_{L^2(\Omega)}\\
\leq & C_0 \Vert \bfU \Vert_{L^\infty(\Omega)} \Vert \nabla w \Vert_{L^2(\Omega)}^2
\end{align*}
Therefore we have 
\[
(k_0 - C_0 \Vert \bfU \Vert_{L^\infty(\Omega)}) \Vert \nabla w \Vert_{L^2(\Omega)}^2 \leq 
\Vert \kappa^{1/2} \nabla w \Vert_{L^2(\Omega)}^2
- \int_\Omega \bfU w \cdot \nabla w.
\]
Since $\bfU$ satisfies \eqref{eq:condU},
it is then easy to conclude that $w = \hat w  = 0$. Thus, we have proved uniqueness, which is equivalent to proving existence.

Next, we show the stability bound \eqref{eq:stability}. We choose $(v_h,\hat v_h) = (c_h^n, \hat c_h^n)$ in \eqref{eq:FEMscheme}-\eqref{eq:DGscheme}. We sum the resulting equations and use coercivity of $a_\Lambda$ and positivity of $b_\Lambda$ to obtain:
\begin{align}
    \frac{1}{\tau}\int_\Omega (c_h^{n}-c_h^{n-1}) c_h^n
      + \frac{1}{\tau} \int_\Lambda |\mathcal{D}|(\hat c_h^{n}- \hat c_h^{n-1}) \hat c_h^n 
      +  (k_0 - C_0\eta_\bfU \Vert \bfU\Vert_{L^\infty(\Omega)}) \Vert \nabla c_h^n \Vert_{L^2(\Omega)}^2
      \nonumber\\
      + C_a \vert \hat c_h^n \vert_{\DG}^2 
      +\Vert \gamma^{1/2} \vert \partial\mathcal{D}\vert^{1/2} \left(\overline{c_h^{n}}-\hat{c}_h^{n}\right)\Vert_{L^2(\Lambda)}^2
%
+ \frac12 \vert \mathcal{D}(0)\vert \hat U (\hat c_h^n(0))^2
%
\nonumber\\
        \leq  \int_\Omega f(\cdot,t^{n}) c_h^n
   + \int_\Lambda \hat f(\cdot, t^{n}) \hat c_h^n
   + |\mathcal{D}(0)|\hat U \, c^{\mathrm{in}}(t^n) \hat c_h^n(0).\label{eq:apriori1}
\end{align}
Next we observe that for the two cases we consider, we have
\begin{equation}\label{eq:twocases}
k_0 - C_0\eta_\bfU \Vert \bfU\Vert_{L^\infty(\Omega)} \geq \frac{k_0}{2}.
\end{equation}
Using Cauchy-Schwarz's inequality we have
\[
\int_\Omega f(\cdot,t^{n}) c_h^n
   + \int_\Lambda \hat f(\cdot, t^{n}) \hat c_h^n
   \leq \frac12 \Vert c_h^n\Vert_{L^2(\Omega)}^2
   + \frac12 \Vert \hat c_h \Vert_{L^2(\Lambda)}^2
   +\frac12 \Vert f(\cdot,t^n)\Vert_{L^2(\Omega)}^2
   + \frac12 \Vert \hat f(\cdot,t^n) \Vert_{L^2(\Lambda)}^2.
\]
We also have with Young's inequality:
\[
|\mathcal{D}(0)|\hat U \, c^{\mathrm{in}}(t^n)  \hat c_h^n(0)
\leq\frac14 |\mathcal{D}(0)|\hat U \,  (\hat c_h^n(0))^2
+ |\mathcal{D}(0)|\hat U \, (c^{\mathrm{in}}(t^n))^2.
\]

Therefore the bound \eqref{eq:apriori1} becomes
\begin{align}
\frac{1}{2\tau}(\Vert c_h^n\Vert_{L^2(\Omega)}^2 - \Vert c_h^{n-1}\Vert_{L^2(\Omega)}^2) 
+\frac{d_0}{2\tau}(\Vert \hat c_h^n\Vert_{L^2(\Lambda)}^2 - \Vert \hat c_h^{n-1}\Vert_{L^2(\Lambda)}^2) 
+  \frac{k_0}{2} \Vert \nabla c_h^n \Vert_{L^2(\Omega)}^2
\nonumber\\
+ C_a \vert \hat c_h^n \vert_{\DG}^2 
+\Vert \gamma^{1/2} \vert \partial\mathcal{D}\vert^{1/2} \left(\overline{c_h^{n}}-\hat{c}_h^{n}\right)\Vert_{L^2(\Lambda)}^2
\leq  
\frac12 \Vert c_h^n\Vert_{L^2(\Omega)}^2
+ \frac12 \Vert \hat c_h \Vert_{L^2(\Lambda)}^2
\nonumber\\
+ \frac12 \Vert f(\cdot,t^n)\Vert_{L^2(\Omega)}^2
+ \frac12 \Vert \hat f(\cdot,t^n) \Vert_{L^2(\Lambda)}^2
+ |\mathcal{D}(0)|\hat U \, (c^{\mathrm{in}}(t^n))^2.
\label{eq:apriori2}
\end{align}
To conclude, we multiply by $2\tau$, sum from $n=1$ to $n=m$,  use the stability of the $L^2$ projection of $\hat c^0$ in the $L^2$ norm, the approximation bound of the Scott-Zhang interpolant of $c^0$ and use Gronwall's lemma. 
\end{proof}

\section{Error analysis}
\label{sec:error}

To handle the coupling term in the analysis below, we will make use of the following result (see Lemma 3.4 in \cite{Zunino19}).
\begin{lemma}\label{lem:trace}
The following inequality holds 
\[
\forall v \in H^1(\Omega), \quad \| \vert \partial \mathcal{D}\vert^{1/2} \overline{v}\|_{L^2(\Lambda)} \leq\|v\|_{L^2(\Gamma)} \leq C_1\|v\|_{H^1(\Omega)}, 
\]
where $C_1$ is the positive constant of the trace inequality from $L^2(\Gamma)$ to $H^1(\Omega)$.
\end{lemma}
Now, under sufficient regularity for the weak solution, we state and prove an error bound that is optimal in time and space with respect to the gradient norm. 
\begin{theorem}
Let $(c,\hat c)$ be the weak solution and assume that 
$c \in L^2(0,T;H^r(\Omega))$, for $1<r<3/2$,  
$\hat c \in L^2(0,T;H^2(\Omega))$, 
$\frac{\partial c}{\partial t} \in L^2(0,T;H^r(\Omega))$ and
$\frac{\partial^2 c}{\partial t^2} \in L^2([0,T]\times\Omega).$
Assume $\tau$ is small enough, namely $\tau \leq \frac12$. 
There exists a constant $M>0$ independent of $h_\Omega, h_\Lambda$ and $\tau$ such that for all $m\geq 1$
    \begin{align}
    \Vert   c(\cdot,t^m)-c_h^m \Vert_{L^2(\Omega)}^2
    +  \Vert \hat c(\cdot,t^m)-c_h^m\Vert_{L^2(\Lambda)}^2
\nonumber\\
+ \tau \sum_{n=1}^m \Vert \nabla c(\cdot,t^n) - \nabla c_h^n \Vert_{L^2(\Omega)}^2
+ \tau \sum_{n=1}^m \vert \hat c(\cdot,t^n) - \hat c_h^n \vert_{\DG}^2
     \leq  
     M (\tau^2 + h_\Omega^{2r-2}+ h_\Lambda^2).
     \end{align}
\end{theorem}
\begin{proof}
We will first obtain error bounds for the errors between the numerical solutions and optimal approximations of the weak solutions. We introduce the notation:
\begin{align*}
\forall n\geq 0, \quad e^n = c_h^n - \Pi c(t^n), \quad
\hat e^n  = \hat c_h^n - \pi \hat c(t^n),\\
\forall t\geq 0, \quad \xi(t) = c(t) - \Pi c(t), \quad
\hat \xi(t)  = \hat c(t) - \pi \hat c(t),
\end{align*}
where for all $t$, $\Pi c(t) \in H^1_0(\Omega)$ is the Scott-Zhang interpolant of $c(t)$ and  $\pi \hat c(t)$ is the L$^2$ projection of $\hat c(t)$ on $V_h^\Lambda$. These approximations satisfy the following optimal bounds \cite{BernardiGiraultETALBook}:
\begin{align}
    & \Vert \xi(t) \Vert_{L^2(\Omega)} \leq M h_\Omega^{r} \Vert c(t)\Vert_{H^r(\Omega)},\,\,
    \Vert \partial_t\xi(t) \Vert_{L^2(\Omega)} \leq M h_\Omega^{r} \Vert \partial_t c(t)\Vert_{H^r(\Omega)}, \label{eq:approxL2c}
    \\
    &\Vert \nabla \xi(t) \Vert_{L^2(\Omega)} \leq M h_\Omega^{r-1} \Vert c(t)\Vert_{H^r(\Omega)},\label{eq:approxgradc}
    \\
    \forall 1\leq i\leq N, & \, \Vert \hat \xi(t)\Vert_{L^2([s_{i-1},s_i])} \leq M h_\Lambda^2 \Vert \hat c(t) \Vert_{H^2(s_{i-1},s_i])},\label{eq:approxL2hatc}\\
    \forall 1\leq i\leq N, & \, \Vert \partial_t \hat \xi(t)\Vert_{L^2([s_{i-1},s_i])} \leq M h_\Lambda^2 \Vert \partial_t \hat c(t) \Vert_{H^2(s_{i-1},s_i])},\label{eq:approxL2hatderc}
    \\
    \forall 1\leq i\leq N, & \, \Vert \frac{\partial\hat \xi(t)}{\partial s} \Vert_{L^2([s_{i-1},s_i])} 
    + h_\Lambda \Vert \frac{\partial^2\hat \xi(t)}{\partial s^2} \Vert_{L^2([s_{i-1},s_i])} 
    \leq M h_\Lambda \Vert \hat c(t) \Vert_{H^2(s_{i-1},s_i])}.\label{eq:approxgradhatc}
\end{align}
The bounds hold for the time derivatives because $\partial_t \Pi c = \Pi \partial_t c$ and $\partial_t \pi c = \pi \partial_t c$.

With the regularity assumptions on the weak solution, it is easy to see that it satisfies for all $n\geq 1$:
\begin{align*}
    \int_\Omega \frac{\partial c}{\partial t}(t^n)  v_h
      +   \int_\Omega  \kappa \nabla c(t^{n}) \cdot \nabla v_h  
      - \int_\Omega \bfU c(t^{n}) \cdot \nabla v_h
      +\int_\Lambda\gamma \vert \partial\mathcal{D}\vert \left(\overline{c(t^{n})}-\hat{c}(t^{n})\right)\overline{v_h}\nonumber\\
        = \int_\Omega f(\cdot,t^{n}) v_h, \quad \forall v_h \in V_h^\Omega,
\end{align*}
and
\begin{align*}
\int_\Lambda |\mathcal{D}| \frac{\partial\hat c(t^n)}{\partial  t} \hat v_h
    + a_\Lambda(\hat c(t^{n}), \hat v_h)
    + b_\Lambda(\hat c(t^{n}), \hat v_h)\nonumber\\
    +\int_\Lambda \gamma|\partial\mathcal{D}|\left(\hat c(t^{n})-\overline{c(t^{n})}\right)   \hat v_h
   =    
   |\mathcal{D}(0)|\hat U \, c^{\text{in}}(t^n) \hat v_h(0)
   + \int_\Lambda \hat f(\cdot, t^{n}) \hat v_h, \quad\forall \hat v_h \in V_h^\Lambda.
\end{align*}
Therefore we have the following error equations
\begin{align}
    \frac{1}{\tau}\int_\Omega (e^{n}-e^{n-1}) v_h
      +   \int_\Omega  \kappa \nabla e^{n} \cdot \nabla v_h
      - \int_\Omega \bfU e^{n} \cdot \nabla v_h
      +\int_\Lambda\gamma \vert \partial\mathcal{D}\vert \left(\overline{e^{n}}-\hat{e}^{n}\right)\overline{v_h}\nonumber\\
     = \int_\Omega \left(\frac{\partial c}{\partial t}(t^n) -\frac{1}{\tau} (\Pi c(t^{n})- \Pi c(t^{n-1})) \right)v_h
      +   \int_\Omega  \kappa \nabla \xi(t^{n}) \cdot \nabla v_h
      - \int_\Omega \bfU \xi(t^{n}) \cdot \nabla v_h\nonumber\\
      +\int_\Lambda\gamma \vert \partial\mathcal{D}\vert \left(\overline{\xi(t^{n})}-\hat{\xi}(t^{n})\right)\overline{v_h},\quad\forall v_h \in V_h^\Omega,
\label{eq:error1}
\end{align}
and
\begin{align}
    \frac{1}{\tau} \int_\Lambda |\mathcal{D}|(\hat e^{n}- \hat e^{n-1}) \hat v_h
    + a_\Lambda(\hat e^{n}, \hat v_h)
    + b_\Lambda(\hat e^{n}, \hat v_h)
    +\int_\Lambda \gamma|\partial\mathcal{D}|\left(\hat e^{n}-\overline{e^{n}}\right)   \hat v_h\nonumber\\
= \int_\Lambda |\mathcal{D}| \left( \frac{\partial \hat c}{\partial t} - \frac{1}{\tau} (\pi \hat c(t^n) - \pi\hat c(t^{n-1}))\right) \hat v_h
    + a_\Lambda(\hat \xi(t^{n}), \hat v_h)
    + b_\Lambda(\hat \xi(t^{n}), \hat v_h)\nonumber\\
    +\int_\Lambda \gamma|\partial\mathcal{D}|\left(\hat \xi(t^{n})-\overline{\xi(t^{n})}\right)   \hat v_h, \quad \forall \hat v_h \in V_h^\Lambda.
\label{eq:error2}
\end{align}
Choosing $v_h = e^n$ in \eqref{eq:error1} and $v_h = \hat e^n$ in \eqref{eq:error2}, summing the two resulting equations, using coercivity of $a_\Lambda$, property \eqref{eq:twocases} and the positivity of $b_\Lambda$, we have
\begin{align*}
    \frac{1}{2\tau} \left(
\Vert   e^n \Vert_{L^2(\Omega)}^2
-\Vert  e^{n-1} \Vert_{L^2(\Omega)}^2 \right)
    + \frac{1}{2\tau} \left(
\Vert \vert \mathcal{D}\vert^{1/2} \hat e^n \Vert_{L^2(\Lambda)}^2
-\Vert \vert \mathcal{D}\vert^{1/2} \hat e^{n-1} \Vert_{L^2(\Lambda)}^2\right)
\nonumber\\
+ \Vert \gamma^{1/2} \vert \partial\mathcal{D}\vert^{1/2} (\overline{e^{n}}-\hat{e}^{n}) \Vert_{L^2(\Lambda)}^2
+\frac{k_0}{2}  \Vert \nabla e^n \Vert_{L^2(\Omega)}^2
+ C_a \vert \hat e^n \vert_{\DG}^2
\nonumber\\
+\frac12 \sum_{i=1}^{N-1}  \ \vert \mathcal{D}(s_i)\vert \hat U ([\hat e^n]_{s_i})^2
+\frac12 \vert \mathcal{D}(0)\vert \hat U (\hat e^n(0))^2
 +\frac12 \vert \mathcal{D}(L)\vert \hat U (\hat e^n(L))^2
 \nonumber\\
     \leq  \int_\Omega \left(\frac{d c}{dt}(t^n) -\frac{1}{\tau} (\Pi c(t^{n})- \Pi c(t^{n-1})) \right) e^n
+  \int_\Lambda |\mathcal{D}| \left( \frac{d \hat c}{d t} - \frac{1}{\tau} (\pi\hat c(t^n) - \pi\hat c(t^{n-1}))\right) \hat e^n
\nonumber\\
 +   \int_\Omega  \kappa \nabla \xi(t^{n}) \cdot \nabla e^n
      - \int_\Omega \bfU \xi(t^{n}) \cdot \nabla e^n 
    + a_\Lambda(\hat \xi(t^{n}), \hat e^n)
    + b_\Lambda(\hat \xi(t^{n}), \hat e^n)\nonumber\\
      +\int_\Lambda\gamma \vert \partial\mathcal{D}\vert \left(\overline{\xi(t^{n})}-\hat{\xi}(t^{n})\right)
 (\overline{e^n} - \hat e^n)
\\
= T_1 + \dots + T_7.
\end{align*}
We now bound each term $T_i$.  For the first two terms $T_1$ and $T_2$, we will use the well-known equality for any function $w$ of one variable $t$:
\[
\frac{1}{\tau}(w(t^n)-w(t^{n-1})
= w'(t^n)-\frac{1}{\tau}\int_{t^{n-1}}^{t^n} (s-t^{n-1}) w''(s) ds.
\]
Using Cauchy-Schwarz's inequality, Young's inequality with a positive $\epsilon_1$ parameter, \eqref{eq:poincare} and the approximation bound \eqref{eq:approxL2c}, we have for $T_1$:
\begin{align*}
T_1 = & \int_\Omega \left(\frac{d c}{dt}(t^n) -\frac{1}{\tau} (\Pi c(t^{n})- \Pi c(t^{n-1})) \right) e^n\\
= & \int_\Omega \left(\frac{d c}{dt}(t^n) -\frac{1}{\tau} (c(t^{n})- c(t^{n-1})) \right) e^n
+ \frac{1}{\tau} \int_\Omega (\xi(t^n)-\xi(t^{n-1}) e^n
\\
\leq &  M \sqrt{\tau} \Vert e^n\Vert_{L^2(\Omega)}
\left\Vert \frac{\partial^2 c}{\partial t^2}\right\Vert_{L^2([t^{n-1},t^n]\times \Omega)}
+ \frac{1}{\tau} \int_\Omega e^n \left(\int_{t^{n-1}}^{t^n}
\frac{\partial \xi}{\partial t}\right)
\\
\leq & \epsilon_1 \Vert \nabla e^n\Vert_{L^2(\Omega)}^2
+ \frac{M}{2 \epsilon_1} \tau \left\Vert \frac{\partial^2 c}{\partial t^2}\right\Vert_{L^2([t^{n-1},t^n]\times \Omega)}^2
+ M \frac{h_\Omega^{2r}}{\epsilon_1 \tau} \left\Vert \frac{\partial c}{\partial t}\right\Vert_{L^2(t^{n-1},t^n;H^r(\Omega))}^2.
\end{align*}
The second term is bounded in a similar fashion as the first term, except we do not use Poincar\'e's inequality. The approximation bound \eqref{eq:approxL2hatderc} and Young's inequality with parameter $\epsilon_2>0$ are used.
\begin{align*}
T_2 
\leq & \epsilon_2 \Vert \hat e^n \Vert_{L^2(\Lambda)}^2
+ \frac{M}{2 \epsilon_2} \tau \left\Vert \frac{\partial^2 \hat c}{\partial t^2}\right\Vert_{L^2([t^{n-1},t^n]\times \Lambda)}^2
+ M \frac{h_\Lambda^{4}}{2\epsilon_2 \tau} \left\Vert \frac{\partial \hat c}{\partial t}\right\Vert_{L^2(t^{n-1},t^n;H^2(\Lambda))}^2.
\end{align*}
The third term is bounded using Cauchy-Schwarz's and
Young's inequalities and the approximation bound \eqref{eq:approxgradc}:
\begin{align*}
T_3 
\leq & \frac{\epsilon_3}{2} \Vert \nabla e^n \Vert_{L^2(\Omega)}^2
+ \frac{1}{2 \epsilon_3} k_1^2 \Vert \nabla \xi(t^n)\Vert_{L^2(\Omega)}^2\\
\leq & \frac{\epsilon_3}{2} \Vert \nabla e^n \Vert_{L^2(\Omega)}^2
+ \frac{M k_1^2}{2 \epsilon_3} h_\Omega^{2r-2} \Vert c(t^n)\Vert_{H^r(\Omega)}^2. 
\end{align*}
Similarly, with Cauchy-Schwarz's, Young's inequalities and the approximation property \eqref{eq:approxL2c}, we have for the fourth term:
\begin{align*}
T_4 
\leq & \frac{\epsilon_4}{2} \Vert \nabla e^n \Vert_{L^2(\Omega)}^2
+ \frac{1}{2 \epsilon_4} \Vert \bfU\Vert_{L^\infty(\Omega)}^2 \Vert \xi(t^n) \Vert_{L^2(\Omega)}^2\\
\leq & \frac{\epsilon_4}{2} \Vert \nabla e^n \Vert_{L^2(\Omega)}^2
+ \frac{M}{2 \epsilon_4} \Vert \bfU\Vert_{L^\infty(\Omega)}^2
h_\Omega^{2r}  \Vert c(t^n)\Vert_{H^r(\Omega)}^2. 
\end{align*}
Using standard DG techniques to bound the terms in the DG bilinear form (\cite{riviere2008discontinuous}, see Appendix~\ref{sec:appe}):
\begin{align*}
T_5 
\leq  \frac{\epsilon_5}{2} 
\vert \hat e^n \vert_{\DG}^2
+ \frac{1}{2 \epsilon_5} M h_\Lambda^2 \Vert \hat c(t^n) \Vert_{H^2(\Lambda)}^2.
\end{align*}
For the sixth term, we have with \eqref{eq:approxL2hatc}, \eqref{eq:approxgradhatc},
a trace inequality, Cauchy-Schwarz's and Young's inequalities:
\begin{align*}
T_6 
= & -\sum_{i=1}^{N} \int_{s_{i-1}}^{s_i}|\mathcal{D}|\hat U \, \hat \xi (t^n)\frac{\partial\hat e^n}{\partial s} 
+\sum_{i=1}^{  N-1 }
    \vert \mathcal{D}(s_i)\vert  \hat U \, \hat \xi(s_i^-,t^n) \, [\hat e^n]_{s_i}\\
    & +|\mathcal{D}(L)| \hat U \, \hat \xi(L, t^n)\, (\hat e^n(L))\\
\leq & 
d_1 \sum_{i=1}^{N} \hat U \Vert \frac{\partial\hat e^n}{\partial s} \Vert_{L^2(s_{i-1},s_i)} \Vert \hat \xi(t^n)\Vert_{L^2(s_{i-1},s_i)}
+ \frac14 \sum_{i=1}^{N-1}  \vert \mathcal{D}(s_i)\vert \hat U [\hat e^n]_{s_i}^2
\nonumber\\
& + \sum_{i=1}^{N-1}  \hat U \vert \mathcal{D}(s_i)\vert \hat \xi(s_i^-,t^n)^2
+\frac14 |\mathcal{D}(L)| \hat U (\hat e^n(L))^2
+ |\mathcal{D}(L)| \hat U \, \hat \xi(L, t^n)^2\\
\leq & \frac{\epsilon_6}{2} \vert \hat e^n \vert_{\DG}^2 
+ \frac{d_1^2}{2\epsilon_6} \hat U^2
\sum_{i=1}^{N} \Vert \hat \xi(t^n)\Vert_{L^2(s_{i-1},s_i)}^2
+ \frac14 \sum_{i=1}^{N-1} \vert \mathcal{D}(s_i)\vert   \hat U [\hat e^n]_{s_i}^2
\nonumber\\
& + \sum_{i=1}^{N-1}  \hat U \vert \mathcal{D}(s_i)\vert \hat \xi(s_i^-,t^n)^2
+\frac14 |\mathcal{D}(L)| \hat U (\hat e^n(L))^2
+ |\mathcal{D}(L)| \hat U \, \hat \xi(L, t^n)^2\\
\leq & \frac{\epsilon_6}{2} \vert \hat e^n \vert_{\DG}^2 
+ M \frac{d_1^2}{2\epsilon_6} \hat U^2
h_{\Lambda}^4 \Vert \hat c(t^n)\Vert_{H^2(\Lambda)}^2
+ \frac14 \sum_{i=1}^{N-1}  \vert \mathcal{D}(s_i)\vert \hat U \, [\hat e^n]_{s_i}^2
\nonumber\\
& +\frac14 |\mathcal{D}(L)| \hat U (\hat e^n(L))^2
+ M d_1  \hat U h_\Lambda^{3} \Vert \hat c(t^n)\Vert_{H^2(\Lambda)}^2.
\end{align*}
For the last term, we first write
\begin{align*}
T_7 
  \leq
 \frac{1}{2}
 \Vert \gamma^{1/2} \vert \partial\mathcal{D}\vert^{1/2} (\overline{e^{n}}-\hat{e}^{n}) \Vert_{L^2(\Lambda)}^2
  + \frac{1}{2}
  \Vert \gamma^{1/2} \vert \partial\mathcal{D}\vert^{1/2} (\overline{\xi(t^{n})}-\hat{\xi}(t^{n})) \Vert_{L^2(\Lambda)}^2
\end{align*}
Using triangle inequality, Lemma~\ref{lem:trace}, and approximation results \eqref{eq:approxL2hatc}, \eqref{eq:approxL2c} and \eqref{eq:approxgradc}, we obtain:
\begin{align*}
\Vert \gamma^{1/2} \vert \partial\mathcal{D}\vert^{1/2} (\overline{\xi(t^{n})}-\hat{\xi}(t^{n})) \Vert_{L^2(\Lambda)}
\leq & \gamma_1 \Vert \vert \partial\mathcal{D}\vert^{1/2} \overline{\xi(t^{n})}
\Vert_{L^2(\Lambda)}
+ \gamma_1\Vert |\partial D|^{1/2} \hat \xi(t^n)\Vert_{L^2(\Lambda)}\\
\leq & \gamma_1 C_1 \Vert \xi(t^n)\Vert_{H^1(\Omega)}
+ \gamma_1 d_1 M h_\Lambda^2 \Vert \hat c(t^n)\Vert_{H^2(\Lambda)}\\
\leq & \gamma_1 C_1 M h_\Omega^{r-1} \Vert c(t^n)\Vert_{H^{r}(\Omega)}
+ \gamma_1 d_1 M h_\Lambda^2 \Vert \hat c(t^n)\Vert_{H^2(\Lambda)}.
\end{align*}
Therefore a bound for $T_7$ is
\[
T_7 \leq \frac{1}{2}
 \Vert \gamma^{1/2} \vert \partial\mathcal{D}\vert^{1/2} (\overline{e^{n}}-\hat{e}^{n}) \Vert_{L^2(\Lambda)}^2
 + M h_\Omega^{2r-2} \Vert c(t^n)\Vert_{H^{r}(\Omega)}^2 + M h_\Lambda^4 \Vert \hat c(t^n)\Vert_{H^2(\Lambda)}^2.
\]
We choose the following values for the parameters $\epsilon_i$:
\[
\epsilon_1 =  \frac{k_0}{12}, \quad
\epsilon_2 = d_0,\quad
\epsilon_3 =
\epsilon_4 =  \frac{k_0}{6}, \quad
\epsilon_5 = 
\epsilon_6 = \frac{C_a}{2},
\]
and combine the bounds above to obtain
\begin{align}
    \frac{1}{2\tau} \left(
\Vert   e^n \Vert_{L^2(\Omega)}^2
-\Vert  e^{n-1} \Vert_{L^2(\Omega)}^2\right)
    + \frac{1}{2\tau} \left(
\Vert \vert \mathcal{D}\vert^{1/2} \hat e^n \Vert_{L^2(\Lambda)}^2
-\Vert \vert \mathcal{D}\vert^{1/2} \hat e^{n-1} \Vert_{L^2(\Lambda)}^2\right)
 \nonumber\\
+ \frac12 \Vert \gamma^{1/2} \vert \partial\mathcal{D}\vert^{1/2} (\overline{e^{n}}-\hat{e}^{n}) \Vert_{L^2(\Lambda)}^2
+\frac{k_0}{4}  \Vert \nabla e^n \Vert_{L^2(\Omega)}^2
+ \frac{C_a}{2} \vert \hat e^n \vert_{\DG}^2
\nonumber\\
+\frac14 \sum_{i=1}^{N-1}  \vert \mathcal{D}(s_i)\vert \hat U ([\hat e^n]_{s_i})^2
+\frac12 \vert \mathcal{D}(0)\vert \hat U (\hat e^n(0))^2
 +\frac14 \vert \mathcal{D}(L)\vert \hat U (\hat e^n(L))^2
 \nonumber\\
     \leq  \Vert \vert \mathcal{D}\vert^{1/2} \hat e^n \Vert_{L^2(\Lambda)}^2
     + M \tau \left( \left\Vert \frac{\partial^2 c}{\partial t^2}\right\Vert_{L^2([t^{n-1},t^n]\times \Omega)}^2
     + \left\Vert \frac{\partial^2 \hat c}{\partial t^2}\right\Vert_{L^2([t^{n-1},t^n]\times \Lambda)}^2 \right)
 \nonumber\\
 + M \frac{h_\Omega^{2r}}{\tau} \left\Vert \frac{\partial c}{\partial t}\right\Vert_{L^2([^{n-1},t^n;H^r(\Omega))}^2
 + M \frac{h_\Lambda^{4}}{\tau} \left\Vert \frac{\partial \hat c}{\partial t}\right\Vert_{L^2(t^{n-1},t^n;H^2(\Lambda))}^2
 \nonumber\\
 + M h_\Omega^{2r-2} \Vert c(t^n)\Vert_{H^r(\Omega)}^2
 + M h_\Lambda^2 \Vert \hat c(t^n)\Vert_{H^2(\Lambda)}^2.
 \label{eq:beforeGronwall}
     \end{align}
The rest of the proof follows a standard technique. After dropping some positive terms in the left-hand side, we multiply \eqref{eq:beforeGronwall} by $2\tau$,
sum from $n=1$ to $n=m$ for any $m\geq 1$, assume that $\tau \leq 1/2$ and apply Gronwall's lemma to obtain:
\begin{align*}
    \Vert   e^m \Vert_{L^2(\Omega)}^2
    + \Vert \vert \mathcal{D}\vert^{1/2} \hat e^m \Vert_{L^2(\Lambda)}^2
+\frac{k_0}{2}  \tau \sum_{n=1}^m \Vert \nabla e^n \Vert_{L^2(\Omega)}^2
+ C_a \tau \sum_{n=1}^m \vert \hat e^n \vert_{\DG}^2
\nonumber\\
     \leq  
\Vert  e^{0} \Vert_{L^2(\Omega)}^2
+\Vert \vert \mathcal{D}\vert^{1/2} \hat e^{0} \Vert_{L^2(\Lambda)}^2 
 + M \tau^2 \left( \Vert \frac{\partial^2 c}{\partial t^2}\Vert_{L^2([0,T]\times \Omega)}^2
     + \Vert \frac{\partial^2 \hat c}{\partial t^2}\Vert_{L^2([0,T]\times \Lambda)}^2 \right)
 \nonumber\\
 + M h_\Omega^{2r-2}\left(
 \Vert \frac{\partial c}{\partial t}\Vert_{L^2(0,T;H^r(\Omega))}^2
+ \Vert c\Vert_{L^2(0,T;H^r(\Omega)}^2\right)
\nonumber\\
 + M h_\Lambda^{2} \left(\Vert \frac{\partial \hat c}{\partial t}\Vert_{L^2(0,T;H^2(\Lambda))}^2
+\Vert \hat c\Vert_{L^2(0,T;H^2(\Lambda)}^2\right).
     \end{align*}
The final result is obtained by first using the approximation bounds \eqref{eq:approxL2c} and \eqref{eq:approxL2hatc} for the initial errors $e^0$ and $\hat e^0$, and second by using triangle inequalities and optimal bounds for the approximation errors.
%

\end{proof}

\section{Numerical results}
\label{sec:numerics}
In this section, we present results for a manufactured solution for a domain with a vertical line and for an unknown solution for a domain with a diagonal line. 

\subsection{Manufactured solution} 
\label{sec:manufactured}

To test convergence, we consider a manufactured solution and compute error rates. The 3D domain is defined as the unit cube centered at the origin, $\Omega = (-0.5,0.5)^3$. The 1D vessel is encompassed within $\Omega$ and is a vertical line $\Lambda = \{(0,0,z), z\in(-0.5,0.5)\}$. Our manufactured solution is  similar to the ones used in  \cite{Masri24,Zunino19}. 
Using polar coordinates, $r = (x^2+y^2)^{1/2}$, we have
\[
\hat c(z,t) = t \left(\sin(\pi z)+2\right), \,\,
    c(x,y,z,t) =  \begin{cases}
        \frac12 \left(1+R\ln\left(\frac{r}{R}\right)\right)\hat{c}(z,t), & r>R, \\
    \frac12 \hat{c}(z,t), & r\leq R.
    \end{cases} 
\]
where $R$ is the radius of cylinder with centerline $\Lambda$. We choose $R=0.05$. 
The other parameters are: $\kappa=\hat\kappa = \gamma = \hat U = 1$.  Since the 1D vessel is aligned with the z-axis, the primary direction of flow is in the vertical direction.  To ensure advection in the 3D domain occurs in the same direction as the 1D domain, we define the 3D velocity as  $\bfU=(0,0,\hat U)$. The prescribed concentration $c^\mathrm{in}$ is deduced from the exact solution $\hat c$.
%
\begin{figure}
    \centering
    \includegraphics[width=1.0\linewidth]{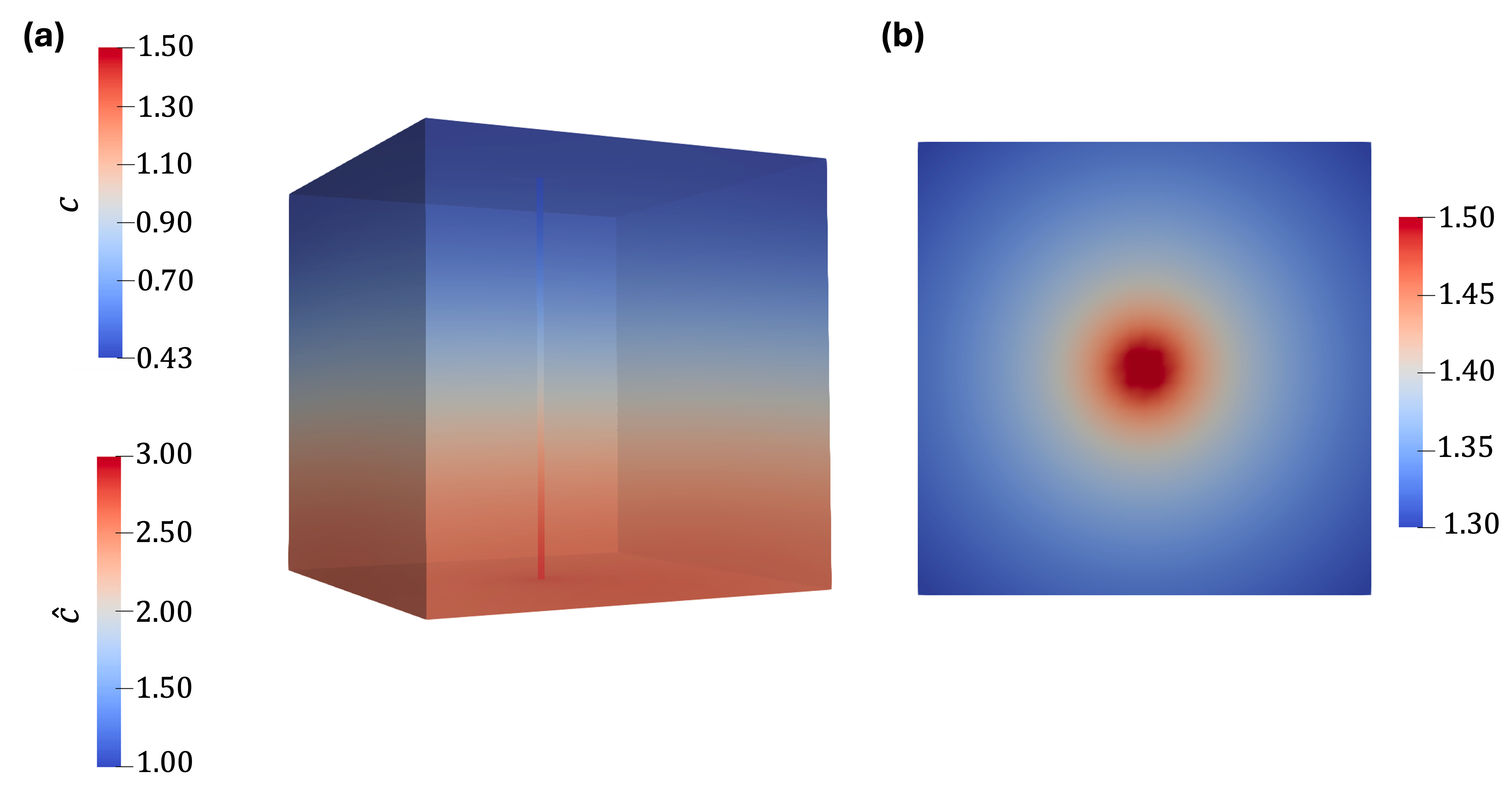}
    \caption{Numerical prediction at $T=1$ of the manufactured solution: left figure shows both concentrations in the whole domain and the right figure shows the concentrations in the plane $z=1/2$. }
    \label{fig:exactsoln}
\end{figure}
We implemented the numerical scheme using FEniCS and   the (FEniCS)$_{ii}$ module that contains the implementation of the lateral average operator \cite{logg2012automated,alnaes2015fenics,kuchta2020assembly}. For the DG parameters, we choose $\epsilon = 1$ and $\sigma=50$. Simulations are ran for a time period of $T=1.0$. Figure~\ref{fig:exactsoln} shows the numerical solutions at $T=1$ in the whole domain (left figure) and in the plane $z=1/2$ (right figure).  We perform a convergence study and compute the errors in the gradient norm and in the $L^2$ norm on successively refined uniform meshes. Table~\ref{tab:convrates_exactsoln3d} 
and Table~\ref{tab:convrates_exactsoln1D} display the errors and their corresponding rates for a mesh size $h_\Omega = h_\Lambda = 1/2^i$ for $2\leq i\leq 6$, at the final time $T=1$.  The time step size is $\tau = 0.1 h_\Omega$.  We observe that the rates in the gradient norms approach the theoretical rate $1/2$ for the 3D solution and $1$ for the 1D solution despite the coupling.  The $L^2$ rates have not converged yet. Optimally they should approach $3/2$ for the 3D solution and $2$ for the 1D solution but there is not theoretical proof of convergence in $L^2$.

\begin{table}[t]
    \centering
    \caption{Numerical errors at $T=t^m=1$ in the gradient norm and $L^2$ norm. }
    \begin{tabular}{|c|c|c|c|c|}
    \hline
        $h_\Omega = h_\Lambda$ & $\Vert \nabla c(\cdot,T)-\nabla c_h^m\Vert_{L^2(\Omega)}$ & Rate & $\Vert c(\cdot,T)- c_h^m \Vert _{L^2(\Omega)}$ & Rate \\
        \hline 
        $1/4$ & $2.5\times 10^{-1}$ & & $1.9\times 10^{-2}$ &    \\
        $1/8$ & $1.4 \times 10^{-1}$ & $0.82$ & $5.4\times 10^{-3}$ & $1.78$ \\
        $1/16$ & $9.1 \times 10^{-2}$ & $0.63$ & $1.7 \times 10^{-3}$ & $1.63$ \\
        $1/32$ & $5.6 \times 10^{-2}$ & $0.71$ & $5.2\times 10^{-4}$ & $1.75$ \\
        $1/64$ & $3.4 \times 10 ^{-2}$ & $0.73$ & $1.4\times 10^{-4}$ & $1.89$ \\
        \hline
    \end{tabular}
    \label{tab:convrates_exactsoln3d}
\end{table}

\begin{table}[t]
    \centering
    \caption{Numerical errors at $T=t^m=1$ in the broken gradient norm and $L^2$ norm.  }
    \begin{tabular}{|c|c|c|c|c|}
    \hline
        $h_\Omega=h_\Lambda$ & $\displaystyle\sum_{i=1}^N \Vert \nabla \hat c(\cdot,T) - \nabla \hat c_h^m\Vert_{L^2([s_{i-1},s_i])}$ & Rate & $\Vert \hat c(\cdot,T) - c_h^m \Vert_{L^2(\Lambda)}$ & Rate \\
        \hline 
      $1/4$ & $5.0 \times 10^{-1}$ & & $4.1 \times 10^{-2}$ & \\
        $1/8$ & $2.5 \times 10^{-1}$ & $0.99$ & $2.3 \times 10^{-2}$ & $0.82$\\
      $1/16$ & $1.3 \times 10^{-1}$ & $0.99$ & $1.3 \times 10^{-2}$  & $0.81$ \\
        $1/32$ & $6.3 \times 10^{-2}$ & $0.99$ & $6.2 \times 10^{-3}$ & $1.07$ \\
        $1/64$ & $3.1 \times 10^{-2}$ & $1.01$ & $2.3 \times 10^{-3}$ & $1.43$ \\
        \hline
    \end{tabular}
    \label{tab:convrates_exactsoln1D}
\end{table}



\subsection{Diagonal line example}
\label{sec:diagline}
We also test convergence without a manufactured solution by defining a diagonal line of length $L$, within the 3D domain and comparing the results on a family of uniform meshes. We let the true solution be defined as the solution obtained with a fine mesh, $h_\Omega = h_\Lambda = h^\ast = 1/64$.  Again, $\Omega = (-0.5,0.5)^3$ and $\Lambda$ is now the diagonal line encompassed within $\Omega$: 
\[
\Lambda = \{(x,y,z), x\in(-0.4,0.4), y\in(-0.4,0.4), z\in(-0.4,0.4)\}.
\]
The final time is $T=1$ and the time step is $\tau=0.1h_\Omega$. The solute is injected at the inlet of the 1D line for $0.1$s with $\hat{c}^{\mathrm{in}} = 5$.
The velocity fields are $\hat{U} = 1$ and $\bfU =\frac{\hat{U}}{\sqrt3}(1,1,1)$, which is the unit tangent vector along the diagonal line.   We implement three different cases:
\begin{itemize}
    \item Case 1: Constant radius $R= 0.05$ and constant $\gamma = 0.1$.
    \item Case 2: Constant $\gamma = 0.1$ and smooth varying radius, 
    \[R(s) = R_{\text{min}}+\frac{R_{\text{max}}-R_{\text{min}}}{2}\left(1+ \tanh\left(\beta\left(\frac{s}{L}-\frac{1}{2}\right)\right)\right),\]
    where $R_{\text{min}}=0.05, R_{\text{max}}=0.08$ are the minimum and maximum radii respectively, and $\beta=8$.
    \item Case 3: Smooth varying radius as defined above and piecewise constant $\gamma(s)$, 
    \[\gamma(s) = \begin{cases}
        0 & 0 \leq s < L/3, \\
        0.05 & L/3 \leq s <2L/3, \\
        0.1 & 2L/3 \leq s \leq L.
    \end{cases}\]
\end{itemize}

From Figure \ref{fig:diagline}, we see the expected behavior in the 3D tissue for all cases shown at different times: $t = 0.0125, t=0.5$ and $t = 1$.  In case 1 there is uniform exchange between the tissue and the vessel. In case 2 we have an increasing radius along the vessel length and we see the band of exchange increases likewise. In case 3, there is no exchange along the part of the vessel where $\gamma=0$ and the solute begins to perfuse into $\Omega$ at $L/3$ where we have $\gamma = 0.05$. Figure \ref{fig:1d_conc} shows the expected behavior of the concentration within the vessel itself. Overall for all three cases, most of the solute has diffused into the surrounding tissue at the final time. Case 2 shows the least amount of solute left in the vessel whereas case 1 shows the most amount left (almost double the amount). This is due to the increase of 3D-1D exchange because of the larger radius value.  There are differences between the three cases regarding the evolution of the concentration, which is shown on the figures at half-time.   At time $t=0.5$, case 3 shows the largest value of the concentration near the inlet, which is consistent with the fact that the first third of the vessel is impermeable with no exchange.
These simulations show the impact of the value of the permeability and radius of the vessel.
\begin{figure}[t]
    \centering
    \includegraphics[width=1.0\linewidth]{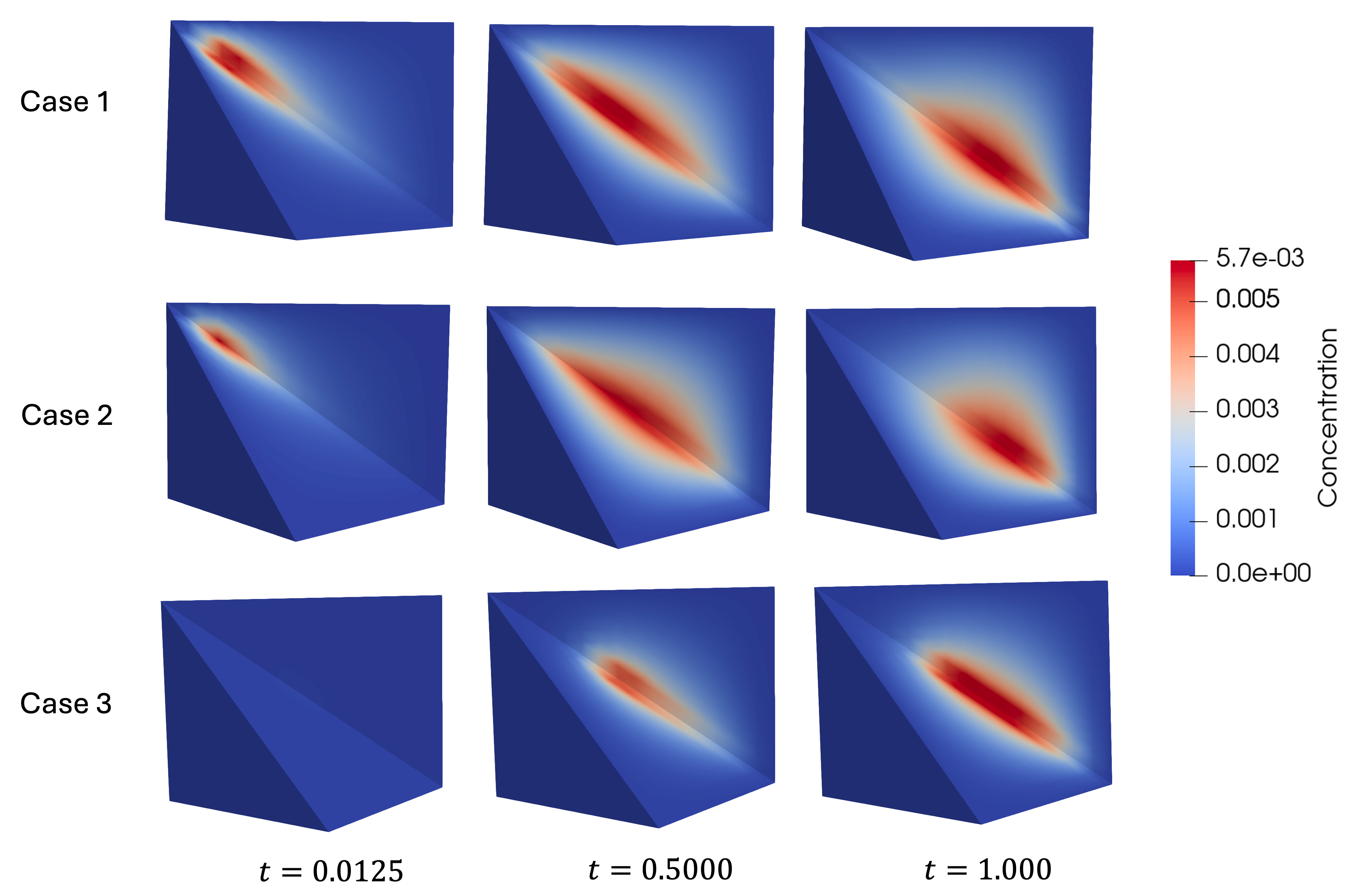}
    \caption{Simulation results of a given concentration administered at the inlet of $\Lambda$, traveling to the outlet and diffusing into $\Omega$ over a time interval $[0,1]$. Case 1: constant radius and global $\gamma$, Case 2: smooth varying radius and global $\gamma$, and Case 3: smooth varying radius and piecewise constant $\gamma$.}
    \label{fig:diagline}
\end{figure}

\begin{figure}
    \centering
    \includegraphics[width=1.0\linewidth]{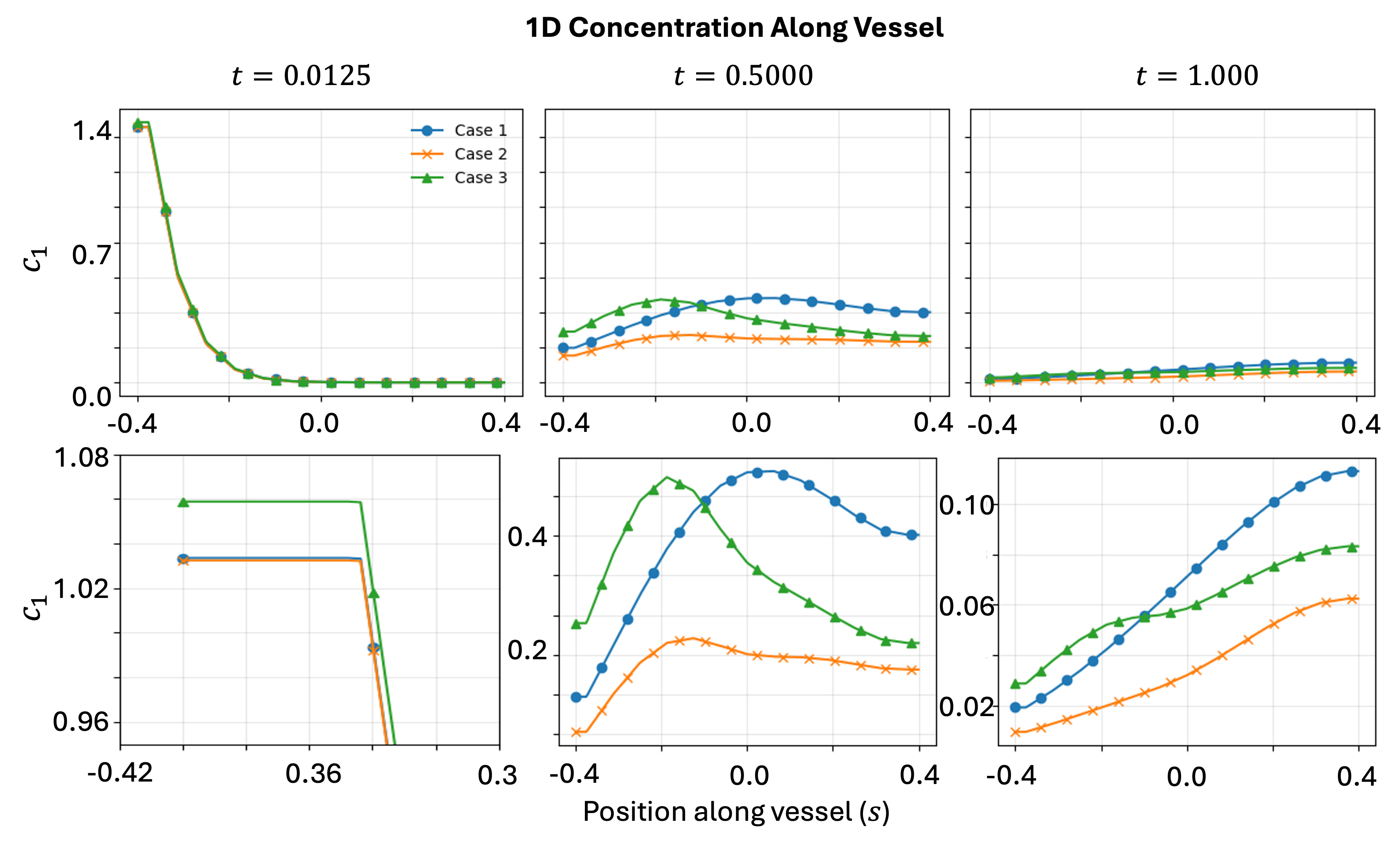}
    \caption{Concentrations $c_1$ plotted along the diagonal vessel defined in Section \ref{sec:diagline} at times $0.0125, 0.005$ and $1$. The top rows shows concentrations at each time in the same scale, while the bottom row has a varying scale to highlight the differences between cases.}
    \label{fig:1d_conc}
\end{figure}
Table \ref{tab:diagconv} shows relative $L^2$ norms of errors for both domains across the three cases. In all cases the errors decrease monotonically as mesh refinement increases, indicating convergence. Although case 3 shows an accelerated convergence to the fine grid solution, the results indicate consistent error reduction across mesh refinement indicating stability and accuracy of the model.  

\begin{table}[h!]
    \centering
    \caption{$L^2$ norms of errors between the numerical solutions obtained on a mesh of size $h_\Omega=h_\Lambda$ and a finer mesh of size $h^\ast = 1/64$. }
    \begin{tabular}{|c|c|c|c|c|c|}
    \hline
    & $h_\Omega = h_\Lambda$ & $\displaystyle\Vert c_h^m - c_{h^\ast}^m \Vert_{L^2(\Omega)}$ & Rate  & $\displaystyle \Vert \hat c_h^m - \hat c_{h^\ast}^m \Vert_{L^2(\Lambda)}$ & Rate \\
        \hline 
         \multirow{4}{3em}{Case 1} &
        $1/4$ & $5.2\times 10^{-1}$  & &  $2.0 \times 10^{-1}$ & \\
        & $1/8$ & $2.2 \times 10^{-1}$ & $1.245$  & $6.5 \times 10^{-2}$ & $1.628$\\
        & $1/16$ & $8.9\times 10^{-2}$ & $1.306$  & $1.7 \times 10^{-2}$  & $1.914$ \\
        & $1/32$ & $2.5\times 10 ^{-2}$ & $1.832$   & $4.1 \times 10^{-3}$ & $2.092$ \\
        \hline
        \multirow{4}{3em}{Case 2} & $1/4$ & $2.6\times10^{-1}$ & & $3.7\times10^{-1}$ & \\
        & $1/8$ & $8.6\times10^{-2}$ & $1.61$ & $9.0\times10^{-2}$ & $2.03$ \\
        & $1/16$ & $1.6\times10^{-2}$ & $2.43$ & $1.5\times10^{-2}$ & $2.56$ \\
        & $1/32$ & $3.5\times10^{-3}$ & $2.16$ & $3.1\times10^{-3}$ & $2.29$ \\
        \hline
        \multirow{4}{3em}{Case 3} & $1/4$ & $1.3\times10^{-1}$ & & $1.2\times10^{-1}$ & \\
        & $1/8$ & $2.5 \times 10^{-2}$ & $2.37$  & $1.3 \times 10^{-2}$ & $3.19$ \\
        & $1/16$ & $2.3\times 10^{-4}$ & $6.82$ & $9.0\times 10^{-6}$ & $10.50$ \\
        & $1/32$ & $2.4 \times 10^{-6}$ & $6.58$ & $4.7 \times 10^{-10}$ & $14.23$\\
        \hline
    \end{tabular}
    \label{tab:diagconv}
\end{table}

\section{Conclusion}
\label{sec:conclusion}

A numerical scheme is analyzed for the coupled 3D-1D transport problem. The concentration in the 3D domain is approximated by continuous piecewise polynomials and the concentration in the 1D domain is approximated by discontinuous polynomials. Error bounds that combine both 3D and 1D errors in $L^2$ in time and gradient in space are optimal with respect to the time step and the regularity of the weak solution. In addition, bounds of the errors in $L^\infty(L^2)$ are also derived.

Because of the use of finite elements for the 3D problem, the velocity field is assumed to be either divergence free or to be bounded by the diffusion coefficient in the 3D domain.  In order to remove these assumptions, the 3D transport problem should be discretized by a discontinuous Galerkin method.  The lack of consistency for the DG method in the 3D domain is a challenge.  This will be the object of future work.

\backmatter

\bmhead{Acknowledgements}

The authors acknowledge partial funding from NSF-DMS 2513092, NSF-DMS 2231482, and NIH 1R01CA301555-01. 


\section*{Declarations}


\begin{itemize}
\item Funding: The authors acknowledge partial funding from NSF-DMS 2513092, NSF-DMS 2231482, NIH 1L40HL181846-01, and NIH 1R01CA301555-01.
 \item Competing Interests: The authors have no competing interests to declare that are relevant to the content of this article.
\item Ethics approval and consent to participate: Not applicable
\item Consent for publication: Not applicable
\item Data availability: Not applicable
\item Materials availability: Not applicable
\item Code availability: Not applicable 
\item Author contribution: Not applicable
\end{itemize}


\begin{appendices}
\section{Derivation of the numerical scheme}
\label{sec:appb}
We consider first the time discretization of \eqref{eq:3D}-\eqref{eq:1D} and obtain at time $t^n$ the semi-discrete equations:
\begin{align*}
\frac{c^n-c^{n-1}}{\tau} & - \nabla \cdot (\kappa \nabla c^n) +\nabla \cdot (\bfU c^n) 
+ \gamma (\overline{c^n} - \hat{c^n}) \delta_\Lambda = f(t^n), \,\, \mbox{in}\,\Omega,\\
|\mathcal{D}(s)|\frac{\hat c^n - \hat c^{n-1}}{\tau} & - \frac{\partial}{\partial s}\left(|\mathcal{D}(s)|\hat \kappa \frac{\partial\hat c^n}{\partial s}\right) +\frac{\partial}{\partial s}\left(|\mathcal{D}(s)| \hat U \hat c^n\right) + |\partial\mathcal{D}(s)|\gamma (\hat c^n-\overline{c^n}) 
= \hat f(t^n), \,\,\mbox{in}\,\Lambda.
\end{align*}
After, we spatially discretize the first equation above by the finite element method, we multiply it by a test function $v_h$ and integrate over $\Omega$.
\begin{align*}
\int_\Omega \frac{c_h^n-c_h^{n-1}}{\tau} v_h
+  \int_\Omega  \kappa \nabla c_h^n \cdot \nabla v_h
- \int_\Omega \bfU c_h^n \cdot \nabla v_h
+\int_\Lambda\gamma \vert \partial\mathcal{D}\vert \left(\overline{c_h^n}-\hat{c}_h^n\right)\overline{v}_h
= \int_\Omega f(\cdot,t^n) v_h.
\end{align*}
For the second equation, we apply the interior penalty discontinuous Galerkin method for spatial discretizations. For notational convenience, $\partial_s \hat v$ denotes the partial derivative of any function $\hat v$ with respect to $s$. Then, we multiply the equation by a test function $\hat v_h$, integrate by parts over one interval $[s_{i-1},s_i]$ and sum over all intervals. 
\begin{align*}
\int_\Lambda |\mathcal{D}|\frac{\hat c_h^n-\hat c_h^{n-1}}{\tau}\hat{v}_h
+\sum_{i=1}^{N} \int_{s_{i-1}}^{s_i} 
|\mathcal{D}|\hat \kappa \partial_s \hat c_h^n\, \partial_s \hat v_h
-  \sum_{i=1}^{N-1} [|\mathcal{D}| \hat \kappa \partial_s \hat c_h^n \, \hat v_h]_{s_i} \\
+ (|\mathcal{D}| \hat \kappa \partial_s \hat c_h^n\, \hat v_h)\Big|_{s_0} - (|\mathcal{D}| \hat \kappa \partial_s \hat c_h^n \, \hat v_h)\Big|_{s_N}
-\sum_{i=1}^{N}\int_{s_{i-1}}^{s_i} \hat{U} |\mathcal{D}| \hat{c}_h^n \, \partial_s \hat v_h +
\sum_{i=1}^{N-1} \hat{U} [ \hat c_h^n |\mathcal{D}|\hat{v}_h]_{s_{i}} \\
+ \hat U (|\mathcal{D}|\hat c_h^n\, \hat{v}_h)\Big|_{s_N}
- \hat U (|\mathcal{D}|\hat c_h^n\, \hat{v}_h)\Big|_{s_0}
+ \int_{\Lambda} \gamma \vert \partial \mathcal{D}\vert (\hat c_h^n - \overline{c_h^n}) \hat v_h
= \int_\Lambda \hat f(\cdot,t^n) \hat v_h.
\end{align*}
Observe that for $1 \leq i \leq N-1$, since the exact solution belongs to $H^2(\Lambda)$
\begin{align}
       [|D| \hat \kappa \partial_s \hat c_h^n \, \hat v_h]_{s_i} = \{ |D|\hat\kappa \, \partial_s \hat c_h^n \}_{s_i} [\hat v_h]_{s_i}, 
       \quad
 [ \hat c_h^n |\mathcal{D}|\hat{v}_h]_{s_{i}}
 = \vert \mathcal{D}(s_i)\vert \hat c_h^n(s_i^-) [ \hat v_h]_{s_i}.
\end{align}
Using the boundary conditions \eqref{eq:bchat1}, \eqref{eq:bchat2}, we rewrite the equation above as:
\begin{align*}
\int_\Lambda |\mathcal{D}|\frac{\hat c_h^n-\hat c_h^{n-1}}{\tau}\hat{v}_h
+\sum_{i=1}^{N} \int_{s_{i-1}}^{s_i} 
|\mathcal{D}|\hat \kappa \partial_s \hat c_h^n\, \partial_s \hat v_h
-  \sum_{i=1}^{N-1} \{ |D|\hat\kappa \, \partial_s \hat c_h^n \}_{s_i} [\hat v_h]_{s_i}\\
     -\sum_{i=1}^{N}\int_{s_{i-1}}^{s_i} \hat{U} |\mathcal{D}| \hat{c}_h^n \, \partial_s \hat v_h +
     \sum_{i=1}^{N-1} \hat{U}  
     \vert \mathcal{D}(s_i)\vert \hat c_h^n(s_i^-) [ \hat v_h]_{s_i} \\
     + \hat U |\mathcal{D}(L)|\hat c_h^n(L) \, \hat{v}_h(L) 
+ \int_{\Lambda} \gamma \vert \partial \mathcal{D}\vert (\hat c_h^n - \overline{c_h^n}) \hat v_h
= \vert \mathcal{D}(0)\vert \hat U c_h^\mathrm{in} v_h(0) + \int_\Lambda \hat f(\cdot,t^n) \hat v_h.
\end{align*}
To conclude we add to the left-hand side of the equation the symmetrization and penalty terms, that vanish for the exact solution.

\section{Proof of coercivity \eqref{eq:coercivity}}
\label{sec:appd}
This proof follows exactly  the one given in  Section 2.7 of \cite{riviere2008discontinuous} and uses
the following trace inequality. There is a positive constant $C_t$ such that
\begin{equation}\label{eq:trace}
\forall \hat v \in \mathbb{P}_k([s_{i-1},s_i]), \,\, \left\vert \frac{\partial \hat v}{\partial s} (s)\right\vert\leq C_t h_i^{-1/2} \Vert  \frac{\partial \hat v}{\partial s} \Vert_{L^2([s_{i-1},s_i])},
\, s  = s_i, s_{i-1}, 
\end{equation}
where $h_i = s_i-s_{i-1}$. 
Using the trace inequality above and Cauchy-Schwarz's inequality, we obtain:
\begin{align*}
  \{ |D(s)|\frac{\partial \hat v}{\partial s}\}_{s_i} [\hat v]_{s_i}  &\leq \Big( \frac{C_t d_1}{2 h^{1/2}_{i}} \Vert \frac{\partial \hat v}{\partial s}\Vert_{L^2([s_{i-1},s_i])} + \frac{C_t d_1}{2 h_{i+1}^{1/2}} \Vert \frac{\partial  \hat v}{\partial s}\Vert_{L^2([s_i,s_{i+1}])} \Big) \Big(\frac{1}{h_\Lambda}\Big)^{\frac{1}{2}-\frac{1}{2}} [\hat v]_{s_i}.
\end{align*}
By summing over $i$, using Cauchy-Schwarz's inequality and Young's inequality (with $\delta>0$), we obtain:
\begin{align*}
    \sum_{i=1}^{N-1}\{ |D(s)|\frac{\partial \hat v}{\partial s}\}_{s_i} [\hat v]_{s_i} &\leq C_td_1\sqrt{2} \Big( \sum_{i=1}^{N-1} \frac{1}{h_\Lambda} [\hat v]^2_{s_i}\Big)^{1/2} \Big( \sum_{i=1}^{N} \Vert \frac{\partial \hat v}{\partial s} \Vert^2_{L^2([s_{i-1},s_i])}\Big)^{1/2} \\
    &\leq  \frac{\delta}{2}\sum_{i=1}^{N} \Vert |D|^{1/2} \frac{\partial \hat v}{\partial s} \Vert^2_{L^2([s_{i-1},s_i])} + \frac{C_t^2 d_1^2 }{\delta d_0}  \sum_{i=1}^{N-1} \frac{1}{h_\Lambda} [\hat v]^2_{s_i}.
\end{align*}
Using the above inequality we obtain a lower bound for $a_\Lambda$:
\begin{align*}
    a_\Lambda (\hat v, \hat v) \geq \Big(1-\frac{\delta}{2}(1+ \epsilon)\Big)
    \sum_{i=1}^{N} \Vert |D|^{1/2} \frac{\partial \hat v}{\partial s} \Vert^2_{L^2([s_{i-1},s_i])}
     + \sum_{i=1}^{N-1} \frac{\sigma - \frac{C_t^2 d_1^2}{\delta d_0}(1+\epsilon)}{h_\Lambda} [\hat v]^2_{s_i}.
\end{align*}
If $\epsilon = 1$ or $0$, we choose $\delta =\frac{1}{2}$ and $\sigma \geq \frac{4C_t^2 d_1^2 }{ d_0} $, thus we have the coercivity result with $C_a = 1/2$. 
The result is trivially correct if $\epsilon = -1$.

\section{Proof of positivity \eqref{eq:bpositivity}}
\label{sec:appc}
%
We write
\begin{align*}
b_\Lambda(\hat v, \hat v) = & 
-\sum_{i=1}^{N} \int_{s_{i-1}}^{s_i}|\mathcal{D}|\hat U \, \hat v\frac{\partial \hat v}{\partial s} 
    +\sum_{i=1}^{  N-1 } \vert\mathcal{D}(s_i)\vert
    \hat U \, \hat v(s_i^-)  [\hat v]_{s_i} 
     +|\mathcal{D}(L)| \hat U \, (\hat v(L))^2.
\end{align*}
Since $\hat U$ is a constant, we write
\[
 -\int_{s_{i-1}}^{s_i}|\mathcal{D}|\hat U\, \hat v \frac{\partial \hat v}{\partial s} = \frac12 \hat U 
 \int_{s_{i-1}}^{s_i} \frac{d|\mathcal{D}|}{ds} \hat v^2
 + \frac12 \hat U \left( \vert \mathcal{D}(s_{i-1}^+)\vert \, \hat v^2(s_{i-1}^+) -
 \vert \mathcal{D}(s_{i}^-)\vert \, \hat v^2(s_{i}^-)\right).
 \]
We thus obtain since $\vert \mathcal{D}\vert$ is continuous and $d(\vert \mathcal{D}\vert)/ds\geq 0$
\begin{align*}
b_\Lambda(\hat v, \hat v) \geq  &
-\sum_{i=1}^{N-1} \frac12 \hat U \vert \mathcal{D}(s_i)\vert \, [\hat v^2]_{s_i}
+\frac12 \vert \mathcal{D}(0)\vert \hat U (\hat v(0))^2\\
& +\frac12 \vert \mathcal{D}(L)\vert \hat U(L) (\hat v(L))^2
 +\sum_{i=1}^{  N-1 } \hat U \vert \mathcal{D}(s_i)\vert  \hat v|_{s_i^-} [\hat v]_{s_i}\\
 \geq  &
\frac12 \sum_{i=1}^{N-1}  \hat U \vert \mathcal{D}(s_i)\vert ([\hat v]_{s_i})^2
+\frac12 \vert \mathcal{D}(0)\vert \hat U (\hat v(0))^2
 +\frac12 \vert \mathcal{D}(L)\vert \hat U(L) (\hat v(L))^2.
\end{align*}
\section{Bound for term $T_5$}
\label{sec:appe}
We want to prove
\begin{align*}
T_5 
\leq  \frac{\epsilon_5}{2} 
\vert \hat e^n \vert_{\DG}^2
+ \frac{1}{2 \epsilon_5} M h_\Lambda^2 \Vert \hat c(t^n) \Vert_{H^2(\Lambda)}^2.
\end{align*}

We expand the term $T_5$:
\begin{align*}
    a_\Lambda(\hat \xi(t^{n}), \hat e^n) &= \sum_{i=1}^{N} \int_{s_{i-1}}^{s_{i}} |D(s)|\frac{\partial \hat \xi(t^{n})}{\partial s}\frac{\partial \hat e^n}{\partial s}- \sum_{i=1}^{N-1} \left\{ |D(s)|\frac{\partial \hat \xi(t^{n})}{\partial s}\right\}_{s_i} [\hat e^n]_{s_i}\\
    &\quad  -\epsilon  \sum_{i=1}^{N-1}\left\{ |D(s)|\frac{\partial \hat e^n}{\partial s}\right\}_{s_i} [\hat \xi(t^{n})]_{s_i} +\sum_{i=1}^{N-1} \frac{\sigma}{h_\Lambda}[\hat \xi(t^{n})]_{s_i} [\hat e^n]_{s_i}.\\
\end{align*}
Using the Cauchy-Schwarz's inequality, Young's inequality and approximation property~\eqref{eq:approxgradhatc}, 
we have:
\begin{align*}
    \left\vert\sum_{i=1}^{N} \int_{s_{i-1}}^{s_{i}} |D(s)|\frac{\partial \hat \xi(t^{n})}{\partial s}\frac{\partial \hat e^n}{\partial s} \right\vert 
     \leq & 
     \frac{\epsilon_5}{4}\sum_{i=1}^N \left\||D|^{\frac{1}{2}}d_s\hat e^n(s)\right\|^2_{L^2([s_{i-1},s_i])}\nonumber\\
&      +
     \frac{1}{\epsilon_5}\sum_{i=1}^N
     \left\||D|^{\frac{1}{2}}\frac{\partial \hat \xi(t^{n})}{\partial s}\right\|^2_{L^2([s_{i-1},s_i])}\\
     \leq  &\frac{\epsilon_5}{4} |\hat e^n|^2_{\DG} + \frac{M}{\epsilon_5} h^2_\Lambda \Vert\hat c(t^n)\Vert^2_{H^2(\Lambda)}.
\end{align*}
For the second term, using triangle inequality, Young's inequality, trace inequality and approximation properties \eqref{eq:approxgradhatc}, we obtain:
\begin{align*}
\left\vert - \sum_{i=1}^{N-1}\{ |D(s)|\frac{\partial \hat \xi(t^{n})}{\partial s}\}_{s_i} [\hat e^n]_{s_i} \right\vert  \leq &
\frac{\epsilon_5}{4}
\sum_{i=1}^{N-1} \frac{\sigma}{h_\Lambda}
[\hat e^n]_{s_i}^2\\
& 
+ \sum_{i=1}^{N-1} \frac{d_1^2 h_\Lambda}{4 \sigma\epsilon_5}  \left( \left\vert \frac{\partial \hat \xi(t^{n})}{\partial s}(s_{i-1}^+)\right\vert^2+
\left\vert \frac{\partial \hat \xi(t^{n})}{\partial s}(s_{i}^-)\right\vert^2\right)   
\\
\leq  & \frac{\epsilon_5}{4} |\hat e^n|^2_\DG + \frac{M}{\epsilon_5}  h^2_\Lambda \Vert \hat c(t^n) \Vert^2_{H^2(\Lambda)}.
\end{align*}
Similarly, we have for the third term with the trace inequality \eqref{eq:trace}, Cauchy-Schwarz's and Young's inequalities and \eqref{eq:approxL2hatc}, \eqref{eq:approxgradhatc}:
\begin{align*}
    \left\vert  \sum_{i=1}^{N-1}\{ |D(s)|\frac{\partial \hat e^n}{\partial s}\}_{s_i} [\hat \xi(t^{n})]_{s_i} \right\vert \leq  &
    \sum_{i=1}^{N-1} |[\hat \xi(t^n)]_{s_i}| C_t \frac{d_1}{2 h_\Lambda^{1/2}} \left( \left\Vert \frac{\partial \hat e^n}{\partial s}\right\Vert_{L^2([s_{i-1},s_i]}
    +\left\Vert \frac{\partial \hat e^n}{\partial s}\right\Vert_{L^2([s_{i},s_{i+1}]}\right) 
    \\
    \leq & 
    \frac{\epsilon_5}{4} |\hat e^n|^2_\DG + \frac{M}{ \epsilon_5} h^3_\Lambda \Vert \hat c\Vert ^2_{H^2(\Lambda)}.
\end{align*}
For the last term, we use \eqref{eq:approxL2hatc} and \eqref{eq:approxgradhatc}:
\begin{align*}
    \left\vert\sum_{i=1}^{N-1} \frac{\sigma}{h_\Lambda}[\hat \xi(t^{n})]_{s_i} [\hat e^n]_{s_i} \right\vert
    &\leq \frac{\epsilon_5}{4}\sum_{i=1}^{N-1}\frac{\sigma}{h_\Lambda}[\hat e^n]_{s_i}^2
    + \frac{1}{\epsilon_5}\sum_{i=1}^{N-1}\frac{\sigma}{h_\Lambda}[\hat \xi(t^{n})]_{s_i}^2 
    \\
    &\leq \frac{\epsilon_5}{4} |\hat e^n|^2_\DG + M h^2_\Lambda \Vert\hat c\Vert^2_{H^2(\Lambda)}.
\end{align*}
By summing each term above, we obtain the desired bound for $T_5$.



\end{appendices}

\bibstyle{sn-mathphys-num} 
\bibliography{bibliography}

\end{document}